\documentclass[11pt]{article}

\usepackage[centertags]{amsmath}
\usepackage{amsfonts}
\usepackage{amssymb}
\usepackage{amsthm}
\usepackage{newlfont}
\usepackage{bibspacing}
\usepackage{verbatim}

\newtheorem{thm}{Theorem}[section]
\newtheorem*{thm*}{Theorem}
\newtheorem{cor}[thm]{Corollary}

\newtheorem{lem}[thm]{Lemma}

\newtheorem{prop}[thm]{Proposition}
\newtheorem*{prop*}{Proposition}

\newtheorem*{conj*}{Conjecture}

\newtheorem*{dfn*}{Definition}
\theoremstyle{definition}
\newtheorem{rem}[thm]{\textbf{Remark}}
\newtheorem*{rmk*}{Remark}
\newtheorem*{fact*}{Fact}

\theoremstyle{proof}

\newcommand{\vol}{\textrm{Vol}}
\newcommand{\norm}[1]{\left\Vert#1\right\Vert}

\newcommand{\abs}[1]{\left\vert#1\right\vert}
\newcommand{\set}[1]{\left\{#1\right\}}
\newcommand{\brac}[1]{\left(#1\right)}
\newcommand{\scalar}[1]{\left \langle #1 \right \rangle}

\newcommand{\Real}{\mathbb{R}}
\newcommand{\Complex}{\mathbb{C}}

\newcommand{\eps}{\varepsilon}
\newcommand{\E}{\mathbb{E}}
\renewcommand{\S}{\mathbb{S}}
\newcommand{\K}{\mathcal{K}}
\newcommand{\I}{\mathcal{I}}

\renewcommand{\P}{\mathbb{P}}
\renewcommand{\H}{\mathcal{H}}

\def \F {\mathcal{F}}

\newlength{\defbaselineskip}
\newcommand{\setlinespacing}[1]           {\setlength{\baselineskip}{#1 \defbaselineskip}}

\numberwithin{equation}{section}

\begin{document}

\title{A Proof of Bobkov's Spectral Bound For Convex Domains via Gaussian Fitting and Free Energy Estimation}
\author{Emanuel Milman\textsuperscript{1}}
\date{}

\footnotetext[1]{Department of Mathematics,
Technion - Israel Institute of Technology, Haifa 32000, Israel. Supported by ISF, BSF, GIF and the Taub Foundation (Landau Fellow).
Email: emilman@tx.technion.ac.il.}

\maketitle

\begin{abstract}
We obtain a new proof of Bobkov's lower bound on the first positive eigenvalue of the (negative) Neumann Laplacian (or equivalently, the Cheeger constant) on a bounded convex domain $K$ in Euclidean space. Our proof avoids employing the localization method or any of its geometric extensions. Instead, we deduce the lower bound by invoking a spectral transference principle for log-concave measures, comparing the uniform measure on $K$ with an appropriately scaled Gaussian measure which is conditioned on $K$. The crux of the argument is to establish a good overlap between these two measures (in say the relative-entropy or total-variation distances), which boils down to obtaining sharp lower bounds on the free energy of the conditioned Gaussian measure. 
\end{abstract}

\section{Introduction}

The following theorem was proved by Sergey Bobkov in \cite{BobkovVarianceBound} (we use the formulation from \cite{EMilman-RoleOfConvexity}, which is formally stronger but ultimately equivalent in the cases of interest):

\begin{thm}[Bobkov] \label{thm:Bobkov}
Let $K$ denote a convex bounded domain in Euclidean space $(\Real^n,\abs{\cdot})$.
Let $X$ denote a random vector uniformly distributed in $K$ (with respect to normalized Lebesgue measure). Then:
\begin{equation} \label{eq:Bobkov}
 D_{Che}(K) \geq \sup_{x_0 \in \Real^n} \frac{c}{\sqrt{\E(|X-x_0|) \S(|X-x_0|)}} ~,
\end{equation}
for some universal constant $c>0$.
\end{thm}

Let us explain the notation used above. We denote the expectation of a random variable $Y$ by $\E(Y)$, and set $\S(Y) := \sqrt{\E((Y - \E(Y))^2)}$ to denote the square root of the variance. We use $D_{Che}(\Omega)$ to denote the Cheeger constant of the domain $\Omega \subset (\Real^n,\abs{\cdot})$, defined as:
\[
D_{Che}(\Omega) := \inf_{A \subset \Omega} \frac{\H^{n-1}(\partial A \cap \Omega)}{\min(\H^{n}(A),\H^{n}(\Omega \setminus A))} ~,
\]
where the infimum ranges over all Borel subsets $A \subset \Omega$, and $\H^k$ denotes the $k$-dimensional Hausdorff measure. When $K$ is a convex domain, it is known by 
results of Maz'ya \cite{MazyaSobolevImbedding,MazyaCheegersInq1} and Cheeger \cite{CheegerInq} on one hand, and Ledoux \cite{LedouxSpectralGapAndGeometry} on the other, that $D_{Che}(K) \simeq \sqrt{\lambda_1(K)}$, where
$\lambda_1(\Omega)$ denotes the first positive eigenvalue of the (negative) Laplacian $-\Delta$ on $\Omega$ with Neumann boundary conditions. Here and elsewhere $A \simeq B$ signifies that there exist universal numeric constants $C_1,C_2 > 0$ so that $C_1 A \leq B \leq C_2 A$. Bobkov's Theorem may therefore be interpreted as a spectral bound for general convex domains. 

\medskip

Of course, classical lower bounds on $\lambda_1(K)$ are known for convex domains $K$, and we mention the sharp lower bound in Euclidean space due to Payne and Weinberger  \cite{PayneWeinberger}:
\begin{equation} \label{eq:PW}
\lambda_1(K) \geq \frac{\pi^2}{diam(K)^2} ~.
\end{equation}
This was extended by Li and Yau \cite{LiYauEigenvalues} and Zhong and Yang  \cite{ZhongYangImprovingLiYau} to locally convex domains with smooth boundary on Riemannian manifolds with non-negative Ricci curvature.
Back to the Euclidean setting, the bound (\ref{eq:PW}) was extended (using the equivalent $D_{Che}(K)$ parameter) by Kannan, Lov\'asz and Simonovits \cite{KLS} to the  bound:
\begin{equation} \label{eq:KLS}
 D_{Che}(K) \geq \sup_{x_0 \in \Real^n} \frac{\log 2}{\E(|X-x_0|)} ~,
\end{equation}
requiring control over the average distance to the centroid $x_0$ rather than control over the diameter $diam(K)$. To be completely accurate, the bound obtained in \cite{KLS} was for a slightly different isoperimetric parameter, which is equivalent to within a factor of $2$ to the $D_{Che}$ parameter; however, a simple modification of the argument (as in \cite[Section 9]{BobkovKappaConcaveMeasures}) yields the asserted (\ref{eq:KLS}). 
The latter bound was extended to the Riemannian setting (with a different numerical constant) in \cite{EMilmanIsoperimetricBoundsOnManifolds}.

By testing very elongated $n$-dimensional cylinders or cones essentially degenerating to one-dimensional densities (and taking the limit as $n \rightarrow \infty$ in the latter case), it is easy to verify that $\pi^2$ and $\log 2$ are the best possible (dimension independent) constants one can use in (\ref{eq:PW}) and (\ref{eq:KLS}), respectively. 
However, assuming that such ``biased" or ``degenerate" domains are prohibited (for instance, a natural non-degeneracy condition is the isotropic condition, requiring that the variance of all unit linear functionals is equal to $1$), Bobkov's bound is in fact the best general bound known to date. Indeed, up to a numeric constant it is certainly better than (\ref{eq:KLS}), since 
$\S(|X-x_0|)^2 \leq \E(|X-x_0|^2) \leq C (\E|X-x_0|)^2$ for some universal constant $C>1$, where the last Khinchine-type inequality is a well known consequence of the convexity of $K$ (as follows from Borell's Lemma \cite{Borell-logconcave}, see \cite{Milman-Schechtman-Book}). And indeed, Bobkov's bound was used in \cite{GuedonEMilmanInterpolating} to deduce the currently best-known estimate on the Cheeger constant for general isotropic convex domains $K$ in $\Real^n$, namely $D_{Che}(K) \geq c n^{-5/12}$; a conjecture of Kannan--Lov\'asz--Simonovits \cite{KLS} asserts that the bound should be $D_{Che}(K) \geq c$, for some dimension independent constant $c > 0$.

\medskip

Bobkov's proof in \cite{BobkovVarianceBound} is based on a localization method,
having its origins in the work of Payne and Weinberger \cite{PayneWeinberger}, rediscovered by Gromov and V. Milman \cite{Gromov-Milman}, and
systematically developed by Kannan--Lov\'asz--Simonovits \cite{LSLocalizationLemma,KLS}.
A more geometric proof of Bobkov's theorem was given in \cite[Theorem 5.15]{EMilman-RoleOfConvexity}, which was based on a general bound on $D_{Che}(K)$ obtained by Kannan--Lov\'asz--Simonovits \cite{KLS} using again localization. The latter general bound was recently given a completely geometric proof in \cite{EMilmanIsoperimetricBoundsOnManifolds}, which generalizes to Riemannian manifolds with non-negative Ricci curvature. 

\medskip

In this work, we propose yet another alternative proof of Bobkov's Theorem \ref{thm:Bobkov} in the Euclidean setting. However, we believe that the method of proof is of independent interest, and leads to interesting questions on the distribution of mass on convex sets. Furthermore, amongst the known proofs of Theorem \ref{thm:Bobkov}, the approach described in this work is the only one which avoids directly employing the localization method (or its isoperimetric geometric extension, developed in \cite{EMilmanIsoperimetricBoundsOnManifolds}).

Our method is based on a comparison between the uniform probability measure on $K$, denoted $\mu_K$, and a suitably scaled Gaussian measure which we condition on $K$:
\[
\gamma_K^w := \exp(Z_K(w)) \exp(-\frac{w}{2} |x|^2) d\mu_K(x) ~, 
\]
where $Z_K(w) \geq 0$ is a normalization term ensuring that $\gamma_K^w$ is a probability measure. When $K$ has unit volume, this measure may be thought of as the Gibbs measure at inverse-temperature $w > 0$ associated with the Hamiltonian:
\[
H(x) := \begin{cases}   \frac{|x|^2}{2} & x \in K \\ +\infty & \text{otherwise}  \end{cases} ~,
\]
and so $\exp(-Z_K(w))$ represents the partition function and $Z_K(w) / w$ represents the system's (Helmholtz) free energy. We have learned from Dario Cordero-Erausquin that the idea of adding $w H(x)$ to a given potential and optimizing on $w > 0$ may be traced back to the work of H\"{o}rmander (e.g. \cite[Theorem 2.2.3]{HormanderL2EstimatesForComplexNeumannProblem}), who derived a generalized version of the Payne--Weinberger estimate (\ref{eq:PW}) for pseudo-convex domains in $\Complex^n$. 

Clearly, it is enough to prove the desired bound (\ref{eq:Bobkov}) for $x_0 = 0$, since $D_{Che}(\Omega)$ is invariant under translation of the domain $\Omega$. We consequently denote $E = E_0 = \E(|X|)$ and $S = S_0 = \S(|X|)$.  It turns out that when $w \leq \frac{c}{E S}$, for an appropriately chosen small constant $c > 0$, $\mu_K$ and $\gamma_K^w$ overlap rather well in the total-variation or relative entropy senses, and general transference principles from \cite{EMilman-RoleOfConvexity,EMilmanGeometricApproachPartII} imply that $\mu_K$ and $\gamma_K^w$ have comparable Cheeger constants (or equivalently, spectral-gaps).  Since $\gamma_K^w$ is well-known to inherit all the isoperimetric and spectral properties of the (non-conditioned) Gaussian measure $\gamma^w := \gamma^w_{\Real^n}$ (e.g. \cite{BakryEmery,BakryLedoux,CaffarelliContraction} or see below), a delicate estimation of the overlap yields the desired bound (\ref{eq:Bobkov}).

Intuitively it is clear that $w$ should be of the order of $\frac{1}{E S}$ to get a good overlap. Indeed, since $\mu_K$ is mostly concentrated (by Chebyshev's inequality) in the annulus $A := K \cap (B(E+2S) \setminus B(E-2S))$ (where $B(R)$ denotes the Euclidean ball of radius $R$), we would like the density of $\gamma_K^w$ to be almost uniform in $A$, yielding the requirement that $w ((E+2S)^2 - (E-2S)^2) \simeq 1$, i.e. that $w \simeq \frac{1}{ES}$. However, we also need to control $Z_K(w)$ to push this argument through, which turns out to be a rather delicate task. This constitutes the main part of this work, and leads to questions on the distribution of $|X|$ in $K$ which may be of independent interest. In particular, the crux of the argument relies on controlling from below the free energy $Z_K(w) / w$ for $w \leq \frac{c}{ES}$.

\medskip

The rest of this work is organized as follows. We recall some preliminaries in Section \ref{sec:2}, and provide the full justification to the intuitive claims made above in Section \ref{sec:3}. We conclude with Section \ref{sec:4}, where we demonstrate that our lower bound for the free energy is in fact sharp.
Unless otherwise stated, all constants $c,c',c'',C,C'$ etc. appearing in this work are universal numeric constants, whose values may change from one occurrence to the next. 

\medskip
\noindent \textbf{Acknowledgement.} I thank Dario Cordero-Erausquin for his interest, discussions and references, and for making me realize that there is something to prove. I also thank the organizers of the 2011 S\'eminaire de Math\'ematiques  Sup\'erieures for the invitation to give a mini-course and to present this work.

\section{Preliminaries} \label{sec:2}

\subsection{Isoperimetry, Spectral-Gap and Concentration} \label{subsec:interactions}

We start by extending our definitions of $D_{Che}$ and $\lambda_1$ to a general metric-measure space setting. A standard text-book reference for most of the notions we employ below is the excellent book by Ledoux \cite{Ledoux-Book} (see also \cite{EMilmanGeometricApproachPartII} for a general overview).  

Let $(\Omega,d,\mu)$ denote a measure-metric space, meaning that $(\Omega,d)$ is a separable metric space and $\mu$ is a Borel probability measure on $(\Omega,d)$. There are various known ways of measuring the interaction between the metric $d$ and the measure $\mu$. One of the strongest forms of measuring this interplay is given by isoperimetric inequalities. Recall that Minkowski's (exterior) boundary measure of a Borel set $A \subset \Omega$, denoted $\mu^+(A)$, is
defined as $\mu^+(A) := \liminf_{\eps \to 0} \frac{\mu(A^d_{\eps}) - \mu(A)}{\eps}$, where $A^d_{\eps} = A^{\Omega,d}_{\eps} := \set{x \in \Omega ; \exists y
\in A \;\; d(x,y) < \eps}$ denotes the $\eps$-neighborhood of $A$ in $(\Omega,d)$.
The isoperimetric profile $\I = \I_{(\Omega,d,\mu)}$ is then defined as the function $\I: [0,1] \rightarrow \Real_+$ given by $\I(v) = \inf \set{ \mu^+(A) ; \mu(A) = v}$. 
An isoperimetric inequality measures the relation between the boundary measure and the measure of a set, by providing a lower bound on $\I_{(\Omega,d,\mu)}$. In this work, we will only be interested in the Cheeger constant $D_{Che}$ of the space $(\Omega,d,\mu)$, defined as:
\[
D_{Che}(\Omega,d,\mu) := \inf_{v \in [0,1]} \frac{\I_{(\Omega,d,\mu)}(v)}{\min(v,1-v)} = \inf_{A \subset \Omega} \frac{\mu^+(A)}{\min(\mu(A),1-\mu(A))} ~,
\]
measuring a certain linear isoperimetric property of the space.
When $d$ and (or) $\mu$ are implied from the context, we simply write $D_{Che}(\Omega)$ or $D_{Che}(\mu)$. We will mostly work in Euclidean space $(\Real^n,|\cdot|)$, and so given a bounded domain $\Omega \subset \Real^n$, we denote $D_{Che}(\Omega) = D_{Che}(\Omega,|\cdot|,\mu_{\Omega})$, where $\mu_{\Omega}$ denotes the uniform probability measure on $\Omega$, and given a Borel probability measure $\nu$, denote $D_{Che}(\nu) = D_{Che}(\Real^n,|\cdot|,\nu)$. 
Note that in the former case, when $A \subset \Omega$ has smooth boundary, then $\mu_\Omega^+(A) = \H^{n-1}(\partial A \cap \Omega) / \H^n(\Omega)$. For the standard Gaussian measure $\gamma$ on $(\Real^n,|\cdot|)$, a classical result of Sudakov--Tsirelson \cite{SudakovTsirelson} and independently Borell \cite{Borell-GaussianIsoperimetry}, asserts that $\I_{\gamma} = \I_{(\Real^n,|\cdot|,\gamma)}= \varphi \circ \Phi^{-1}$, where $\varphi(t) := \frac{1}{\sqrt{2 \pi}} \exp(-t^2/2)$ and $\Phi(t) := \int_{-\infty}^t \varphi(s) ds$, corresponding to the fact that half-planes are isoperimetric minimizers of Gaussian boundary measure. It is elementary to verify that $\I_{\gamma}$ is a concave function on $[0,1]$ vanishing at the endpoints and symmetric about the point $1/2$, and so consequently $D_{Che}(\gamma) = 2 \I_{\gamma}(1/2) = \sqrt{2 / \pi}$, and by scaling $D_{Che}(\gamma^w) = \sqrt{2 / \pi} \sqrt{w}$ (see \cite{EMilman-RoleOfConvexity,EMilmanGeometricApproachPartI} for the concavity of the isoperimetric profile in more general spaces). 

Another way of measuring the interaction between metric and measure is given by Sobolev inequalities. In this work, we will only be interested in the Poincar\'e inequality. 
Let $\F = \F(\Omega,d)$ denote the space of functions which are Lipschitz on every ball in $(\Omega,d)$. Given $f \in \F$, define $|\nabla f|$ as the following Borel function:
\[
 \abs{\nabla f}(x) := \limsup_{d(y,x) \rightarrow 0+} \frac{|f(y) - f(x)|}{d(x,y)} ~.
\]
(and we define it as 0 if $x$ is an isolated point - see \cite[pp. 184,189]{BobkovHoudre} for more details). In the smooth Euclidean setting, $\abs{\nabla f}$ of course coincides with the Euclidean length of the gradient of $f$. We say that $(\Omega,d,\mu)$ satisfies a Poincar\'e inequality if:
\[
\exists \lambda_1 > 0 \;\;\; \forall f \in \F \;\;\; \int |\nabla f|^2 d\mu \geq \lambda_1 \brac{\int f^2 d \mu - \brac{\int f d\mu}^2} ~.
\]
The best possible constant $\lambda_1$ above, called the Poincar\'e or spectral-gap constant, is denoted $\lambda_1(\Omega,d,\mu)$. As usual, we will use $\lambda_1(\Omega)$ or $\lambda_1(\mu)$ when the space is implied from the context. Note that when $\Omega$ is a smooth domain in Euclidean space, $\lambda_1(\Omega)$ coincides with the first positive eigenvalue of the (negative) Laplacian $-\Delta$ on $\Omega$ with Neumann boundary conditions. For instance, it is well known (e.g. \cite{Ledoux-Book}) for the standard Gaussian measure $\gamma$ on $(\Real^n,|\cdot|)$ that $\lambda_1(\gamma) = 1$, and hence by scaling $\lambda_1(\gamma^w) = w$.

As already alluded to in the Introduction, there is an intimate relation between $D_{Che}$ and $\lambda_1$. It was shown by Cheeger \cite{CheegerInq} and independently by Maz'ya \cite{MazyaSobolevImbedding,MazyaCheegersInq1} that $\sqrt{\lambda_1(\Omega,d,\mu)} \geq \frac{1}{2} D_{Che}(\Omega,d,\mu)$ (their proof extends from the Euclidean or Riemannian setting to the general metric-measure space one, cf. \cite{EMilman-RoleOfConvexity}). The converse inequality is in general false, due to the possible existence of narrow ``necks" in the space (see e.g. \cite{EMilman-RoleOfConvexity}). However, it was shown by Buser \cite{BuserReverseCheeger} for the case of a closed Riemannian manifold having \emph{non-negative} Ricci curvature, and extended by Ledoux \cite{LedouxSpectralGapAndGeometry} to the case of a manifold-with-density having \emph{non-negative} generalized Ricci curvature, that the converse inequality also holds up to a universal numeric constant. We do not provide unnecessary definitions here and only mention that in particular, in the case of Euclidean space $(\Real^n,|\cdot|)$ endowed with a log-concave probability measure $\mu$,  it follows that $\sqrt{\lambda_1(\mu)} \simeq D_{Che}(\mu)$.  
Recall that a measure $\mu$ on $\Real^n$ is called log-concave if $\mu = \exp(-V(x)) dx$ with $V : \Real^n \rightarrow \Real \cup \set{+\infty}$ convex (cf. Borell \cite{Borell-logconcave}). Note that all of the probability measures mentioned in the Introduction $\mu_K$, $\gamma^w_K$ and $\gamma^w$ are log-concave. 

A third way of measuring the interaction between metric and measure is given by concentration inequalities. The concentration profile of our space $\K = \K_{(\Omega,d,\mu)} :\Real_+ \rightarrow [0, 1/2]$ is defined as the pointwise minimal function so that $1 - \mu(A^d_r) \leq \K(r)$ for all Borel sets $A \subset \Omega$ with $\mu(A) \geq 1/2$ ($r \geq 0$). Clearly $\K$ is a non-increasing function.  Concentration inequalities measure how tightly the measure $\mu$ is concentrated around sets having measure $1/2$ as a function of the distance $r$ away from these sets, by providing an upper bound on $\K(r)$ which decays to $0$ as $r \rightarrow \infty$. To measure this decay, we denote by $D_{Exp_p} = D_{Exp_p}(\Omega,d,\mu) \geq 0$ the best constant in the following inequality ($p > 0$):
\[
\forall r \geq 0 \;\;\; \K(r) \leq \exp( - (D_{Exp_p} r)^p) ~.
\]
When $D_{Exp_1} > 0$ ($D_{Exp_2} > 0$) we will say that our space has exponential (Gaussian) concentration (respectively).  Indeed, it is well-known (see e.g. \cite{Ledoux-Book}) that the standard Gaussian measure $\gamma$ on $(\Real^n,|\cdot|)$ satisfies $D_{Exp_2}(\gamma) = 1/\sqrt{2}$ , and hence by scaling $D_{Exp_2}(\gamma^w) = \sqrt{w/2}$. Furthermore, it was shown by M. Gromov and V. Milman \cite{GromovMilmanLevyFamilies} that $D_{Exp_1} \geq c \sqrt{\lambda_1}$ for some universal numeric constant $c > 0$, i.e. that having positive spectral-gap implies exponential concentration  (their proof applies in a general metric-measure space setting). The converse inequality is again in general false (consider e.g. a measure with disconnected support). However, for the class of log-concave measures 
in Euclidean space, it was shown in our previous work \cite{EMilman-RoleOfConvexity} that in fact $D_{Exp_1} \geq c' D_{Che}$ for some universal constant $c' > 0$, which implies together with the above mentioned results that for this class $D_{Che} \simeq \sqrt{\lambda_1} \simeq D_{Exp_1}$ (see \cite{EMilman-RoleOfConvexity} for a much stronger result in a more general setting). 

\subsection{Stability of $D_{Che}$ in the class of log-concave measures}

Next, we will require two notions of proximity between two Borel probability measures $\mu_1$ and $\mu_2$ on $(\Omega,d)$. The first is given by the total-variation distance, denoted $d_{TV}(\mu_1,\mu_2)$ and defined as:
\[
d_{TV}(\mu_1,\mu_2) := \sup_{A \subset \Omega}  \abs{\mu_1(A) - \mu_2(A)} ~.
\]
The second is given by the relative entropy (or Kullback--Leibler divergence) of $\mu_2$ with respect to $\mu_1$, denoted $H(\mu_2 | \mu_1)$ and defined as:
\[
H(\mu_2 | \mu_1) := Ent_{\mu_1}\brac{\frac{d\mu_2}{d\mu_1}} = \int \log\brac{\frac{d\mu_2}{d\mu_1}} d\mu_2 ~,
\]
if $\mu_2 \ll \mu_1$, and $\infty$ otherwise.

The key ingredient in our proof of Bobkov's Theorem \ref{thm:Bobkov} is the following theorem, compiled from our previous results from \cite{EMilman-RoleOfConvexity,EMilmanGeometricApproachPartII}:

\begin{thm}[\cite{EMilman-RoleOfConvexity,EMilmanGeometricApproachPartII}] \label{thm:transfer}
Let $\mu_1,\mu_2$ denote two log-concave probability measures on $(\Real^n,|\cdot|)$. 
\begin{itemize}
\item  There exists a universal constant $c > 0$ so that for any $\eps > 0$:
\[
d_{TV}(\mu_1,\mu_2) \leq 1 - \eps \;\;\; \Rightarrow \;\;\; D_{Che}(\mu_2) \geq c \frac{\eps^2}{\log(1/\eps)} D_{Che}(\mu_1) ~.
\]
\item There exists a universal constant $c > 0$ so that for any $p \geq 1$:
\[
H(\mu_2 | \mu_1) \leq L \;\;\; \Rightarrow \;\;\; D_{Che}(\mu_2) \geq \frac{c}{1+L^{1/p}} D_{Exp_p}(\mu_1) ~. 
\]
Since $D_{Exp_1}(\mu_1) \geq c' D_{Che}(\mu_1)$, in particular:
\[
H(\mu_2 | \mu_1) \leq L \;\;\; \Rightarrow \;\;\; D_{Che}(\mu_2) \geq \frac{c c'}{1+L} D_{Che}(\mu_1) ~. 
\]
\end{itemize}
\end{thm}

\subsection{Conditioning the Gaussian measure onto $K$}

Recall that given $w > 0$, we define $\gamma_K^w$ by conditioning the Gaussian measure $\gamma^w$ onto $K$, namely:
\begin{equation} \label{eq:gamma-def}
\gamma_K^w := \exp(Z_K(w)) \exp\brac{-\frac{w}{2} |x|^2} d\mu_K(x) ~,
\end{equation}
where $\mu_K$ denotes the uniform (probability) measure on $K$, and $Z_K(w) \geq 0$ is a normalization term ensuring that $\gamma_K^w$ is a probability measure. The following estimates are well-known:
\begin{thm} \label{thm:estimates}
\[
\forall w > 0 \;\;\;\; D_{Che}(\gamma_K^w) \geq \sqrt{\frac{2}{\pi}} \sqrt{w} ~,~ \lambda_1(\gamma_K^w) \geq w ~,~ D_{Exp_2}(\gamma_K^w) \geq \sqrt{w/2} ~.
\]
\end{thm} 
\noindent For the sake of completeness, we present (several variants of) the proof. 
\begin{proof}
One way to obtain all these estimates simultaneously and with the optimal constants is by applying Caffarelli's Contraction Theorem. It was shown in \cite{CaffarelliContraction} that given a source probability measure $\mu_1$ and target probability measure $\mu_2$ on $\Real^n$, having the form $\mu_1 = \gamma^w$ and $\mu_2 = f \gamma^w$ with $f$ log-concave, there exists a Borel map $T : (\Real^n,|\cdot|) \rightarrow (\Real^n,|\cdot|)$ pushing forward $\mu_1$ onto $\mu_2$ which contracts Euclidean distance: $|T(x) - T(y)| \leq |x-y|$ for all $x,y \in \Real^n$. In fact, $T$ is the Brenier optimal-transport map \cite{BrenierMap}, minimizing the $L^2$-average transport cost $\int |T(x) - x|^2 d\mu_1(x)$ among all maps pushing forward $\mu_1$ onto $\mu_2$ (see also \cite{KimEMilmanGeneralizedCaffarelli} for a different construction in a more general setup). We note that Caffarelli's result applies regardless of the smoothness of $f$, since we can always approximate $f$ by smooth log-concave functions $\set{f_n}$ so that $\set{f_n \mu_1}$ converge in total-variation to $f \mu_1$, and the corresponding contractions $\set{T_n}$ pushing forward $\mu_1$ onto $\set{f_n \mu_1}$ will necessarily converge (by the Arzel\`a--Ascoli Theorem) to our desired contraction $T$ (see \cite[Lemma 3.3]{KimEMilmanGeneralizedCaffarelli} for the precise argument). Note that the convexity of $K$ ensures that $f_0 := \frac{d \gamma_K^w}{d\gamma^w} = c_{K,w} \mathbf{1}_K$ is log-concave, and so by Caffarelli's Contraction Theorem there exists a contraction pushing forward $\gamma^w$ onto $\gamma^w_K$. 

It is easy to verify that all of forms of interaction between metric and measure described in Subsection \ref{subsec:interactions} can only improve under contracting maps. Namely, if $T : (\Omega_1,d_1,\mu_1) \rightarrow (\Omega_2,d_2,\mu_2)$ is a Borel map pushing forward $\mu_1$ onto $\mu_2$ and contracting distances $d_2(T(x),T(y)) \leq d_1(x,y)$ for all $x,y \in \Omega_1$, then $\I_{(\Omega_2,d_2,\mu_2)} \geq \I_{( \Omega_1,d_1,\mu_2)}$, $\abs{ \nabla (f \circ T)}(x) \leq \abs{\nabla f} (T x)$ and $\K_{(\Omega_2,d_2,\mu_2)} \leq \K_{( \Omega_1,d_1,\mu_2)}$. In particular, it immediately follows that:
\[
D_{Che}(\gamma_K^w) \geq D_{Che}(\gamma^w) ~,~ \lambda_1(\gamma_K^w) \geq \lambda_1(\gamma^w) ~,~ D_{Exp_2}(\gamma_K^w) \geq D_{Exp_2}(\gamma^w) ~,
\]
and the asserted bounds follow from the well-known bounds for the Gaussian measure $\gamma^w$ mentioned in Subsection \ref{subsec:interactions}. 

Another possible way to derive the above mentioned bounds (with perhaps inferior numerical constants) is as follows. The Brascamp--Lieb inequality \cite{BrascampLiebPLandLambda1} states that if $\mu = \exp(-V(x)) dx$ is a log-concave measure on $\Real^n$ with $V \in C^2(\Real^n)$ having positive definite Hessian $Hess \; V > 0$, then for any $f \in C^1(\Real^n)$:
\[
\int \scalar{ (Hess \; V)^{-1} \nabla f , \nabla f} d\mu \geq \int f^2 d\mu - \brac{\int f d\mu}^2 ~.
\] 
In particular, it follows that if $Hess \; V \geq \lambda Id$ with $\lambda > 0$ then $\lambda_1(\mu) \geq \lambda$ (the latter statement may also be attributed to Lichnerowicz, at least in an analogous Riemannian setting, cf. \cite[Theorem 4.70]{GHLBookEdition1}). Moreover, if $Hess \; V \geq \lambda Id$ with $\lambda > 0$, then by the Bakry--\'Emery criterion \cite{BakryEmery}, $(\Real^n,|\cdot|,\mu)$ satisfies in fact a stronger log-Sobolev inequality, and the Herbst argument implies that $D_{Exp_2}(\mu) \geq c \sqrt{\lambda}$ for some universal constant $c>0$ (the interested reader may consult \cite{Ledoux-Book,EMilmanGeometricApproachPartII} for missing definitions and arguments). The inequality $D_{Che}(\mu) \geq c' \sqrt{\lambda}$ may be deduced by using that $D_{Che}(\mu) \simeq \sqrt{\lambda_1(\mu)}$ for log-concave measures $\mu$; alternatively, one may directly deduce that $D_{Che}(\mu) \geq D_{Che}(\gamma^\lambda)$ by applying the isoperimetric inequality of Bakry--Ledoux \cite{BakryLedoux}. Finally, an approximation argument as the one described above ensures that all isoperimetric, Sobolev and concentration estimates pass onto $\gamma^w_K = f_0 \gamma_w$ (see \cite{EMilman-RoleOfConvexity,EMilmanGeometricApproachPartI,EMilmanRoleOfConvexityInFunctionalInqs} for more technical approximation results, which in fact apply in the more challenging Riemannian setting). 
\end{proof}

\section{Proof of Theorem \ref{thm:Bobkov}} \label{sec:3}

In this section we provide a proof of Bobkov's Theorem \ref{thm:Bobkov} via Gaussian Fitting and estimation of the free energy $Z_K(w) / w$.

We denote $E_2 = \sqrt{\E(|X|^2)}$. Clearly $E \leq E_2$ by Jensen's inequality, but as mentioned in the Introduction, it is also well-known that the convexity of $K$ ensures that $E_2 \leq C E$ for some universal numeric constant $C > 1$. Together with the obvious reduction to the case $x_0 = 0$ described in the Introduction, the proof of Theorem \ref{thm:Bobkov} thus reduces to proving the following equivalent bound:

\begin{equation} \label{eq:Bobkov2}
 D_{Che}(K) \geq \frac{c}{\sqrt{E_2 S}} ~.
\end{equation}

\subsection{Reduction to Estimation of the Free Energy}

Note that the normalization term $Z_K(w) \geq 0$ appearing in (\ref{eq:gamma-def}) is given by:
\[
Z_K(w) := - \log \int \exp(-\frac{w}{2} |x|^2) d\mu_K(x) ~.
\]
We summarize several useful properties of $Z_K(w)$ below:
\begin{lem}
\hfill
\begin{itemize}
\item
The function $[0,\infty) \ni w \mapsto Z_K(w)$ is concave increasing. 
\item
$Z_K(0) = 0$ and $\frac{d}{dw}Z_K(0) =  \frac{1}{2} E_2^2$.
\item $Z_K(w) = \frac{n}{2} \log(w / (2\pi )) - \log \gamma(\sqrt{w} K)$, 
where $\gamma$ denotes the standard Gaussian measure on $\Real^n$. 
\item
$(0,\infty) \ni w \mapsto Z_K(w) / w$ decreases continuously from $\frac{1}{2} E_2^2$ to $0$.
\end{itemize}
\end{lem}
\begin{proof}
The first assertion follows immediately by direct differentiation of by employing H\"{o}lder's inequality. The second and third ones are immediate to verify as well. The fourth one is an immediate consequence of the first three. 
\end{proof}

\begin{rem}
Some properties and extremal characterizations of $\gamma(K)$ for symmetric convex sets $K$ in $\Real^n$ have been studied by Cordero-Erausquin, Fradelizi and Maurey \cite{CFM-BConjecture} and by Bobkov \cite{BobkovFloridaProceedings}.  In view of the third assertion of the above lemma, it is immediate to translate these properties into statements about $Z_K(w)$. However, we have not found any concrete applications of these connections, and so we only vaguely mention these in passing, and refer the interested reader to the above mentioned references. 
\end{rem}

To transfer the lower bound on $D_{Che}(\gamma_K^w)$ given in Theorem \ref{thm:estimates}, namely $D_{Che}(\gamma_K^w) \geq c \sqrt{w}$, onto a lower bound on $D_{Che}(\mu_K)$, we would like to invoke one of the transference principles given by Theorem \ref{thm:transfer}. To this end, we must control the total-variation distance $d_{TV}(\mu_K,\gamma^w_K)$ or relative entropy $H(\mu_K | \gamma^w_K)$. The calculation is a little more concise for the latter option, so we proceed with the task of bounding the relative entropy from above.

\begin{lem}
\hfill
\begin{itemize}
\item
\begin{eqnarray} 
\nonumber H(\mu_K | \gamma_K^w) &=& \int \brac{\frac{w}{2} |x|^2 - Z_K(w)} d\mu_K(x) \\
\label{eq:H-formula} &=& \frac{1}{2} E_2^2 w - Z_K(w) = w \brac{\frac{1}{2} E_2^2 - \frac{Z_K(w)}{w}}  ~.
\end{eqnarray}
\item The function $[0,\infty) \ni w \mapsto H(\mu_K | \gamma_K^w)$ is convex and increases from $0$ to $\infty$.  
\end{itemize}
\end{lem}
\begin{proof}
The first assertion follows by definition. The second assertion is a direct consequence of the first one and the concavity of $Z_K(w)$. 
\end{proof}

Since the measure $\gamma_K^w$ has Gaussian concentration by Theorem \ref{thm:estimates}, namely $D_{Exp_2}(\gamma_K^2) \geq \sqrt{w /2}$, the transference principle for the relative entropy given by Theorem \ref{thm:transfer} guarantees that if:
\[
H(\mu_K | \gamma_K^{w})  \leq w L_w ~,
\]
then for some universal constant $c' >0$:
\begin{equation} \label{eq:optimize}
D_{Che}(\mu_K) \geq c' \frac{D_{Exp_2}(\gamma_K^w)}{1 + \sqrt{w L_w}} \geq \frac{c' / \sqrt{2}}{1/\sqrt{w} + \sqrt{L_w}} ~.
\end{equation}
Optimizing on $w > 0$ in (\ref{eq:optimize}), we see that we would like to choose $w_0>0$ as large as possible so that $H(\mu_K | \gamma_K^{w_0}) \leq w_0 L_{w_0} \simeq 1$, 
resulting in the bound $D_{Che}(\mu_K) \geq c'' \sqrt{w_0}$. Recalling (\ref{eq:H-formula}), this boils down to choosing $w_0 > 0$ as large as possible so that:
\begin{equation} \label{eq:goal}
\frac{Z_K(w_0)}{w_0} \geq \frac{1}{2} E_2^2 -  \frac{C}{w_0} ~,
\end{equation}
for some numeric constant $C>0$.  We will show that (\ref{eq:goal}) is indeed satisfied with $w_0 = c / E_2 S$, for some appropriately small numeric constant $c > 0$, from whence Bobkov's bound (\ref{eq:Bobkov2}) follows.

\begin{thm}[Free Energy Lower Bound] \label{thm:main-calc}
There exists universal constants $c,C>0$ so that for any convex bounded domain $K$ in $(\Real^n,|\cdot|)$, we have the following estimate for the free energy:
\begin{equation} \label{eq:main-calc}
w \leq \frac{c}{E_2 S} \;\;\; \Rightarrow \;\;\; \frac{Z_K(w)}{w} \geq \frac{1}{2} E_2^2 - C E_2 S ~.
\end{equation}
\end{thm}

Before proceeding with the proof of Theorem \ref{thm:main-calc}, we summarize below its consequences as described in the preceding discussion:

\begin{cor}
There exist two numeric constants $c',c'' > 0$ so that for any convex bounded domain $K$ in $(\Real^n,|\cdot|)$: 
\begin{itemize}
\item Setting $w_0 = c' / (E_2 S)$, we have $H(\mu_K | \gamma_K^{w_0})  \leq 1/2$ and $d_{TV}(\mu_K , \gamma_K^{w_0}) \leq 1/2$.
\item Bobkov's bound: $D_{Che}(K) \geq c'' / \sqrt{E_2 S}$. 
\end{itemize}
\end{cor}
\begin{proof}
Theorem \ref{thm:main-calc} and (\ref{eq:H-formula}) guarantee that setting $c ' = \min(c,1/(2C))$, where $c,C>0$ are given in Theorem \ref{thm:main-calc}, the bound on  $H(\mu_K | \gamma_K^{w_0})$ follows. The bound on $d_{TV}(\mu_K , \gamma_K^{w_0})$ follows by the Csisz\'ar--Kullback--Pinsker inequality (e.g. \cite{Ledoux-Book}), stating that for any two Borel probability measures $\mu_1,\mu_2$:
\[
d_{TV}(\mu_1,\mu_2) \leq \sqrt{ H(\mu_2 | \mu_1) / 2} ~.
\] 
Bobkov's lower bound on $D_{Che}(K)$ follows by the discussion preceding the statement of Theorem \ref{thm:main-calc}, employing the transference principle given by Theorem \ref{thm:transfer} and the well-known bound $D_{Che}(\gamma_K^{w})\geq c''' \sqrt{w}$ stated in Theorem \ref{thm:estimates}. 
\end{proof}

\subsection{Main Calculation}

For the proof of Theorem \ref{thm:main-calc}, we require the following:
\begin{lem} \label{lem:lc}
There exists a universal constant $c_1 := \log(3)/4 > 0$ so that for any convex bounded domain $K$ in $(\Real^n,|\cdot|)$:
\[
\mu_K\set{|x| \leq r } \leq \exp\brac{-c_1 \frac{E_2 - r}{S}}  \;\;\;  \forall r \in [0,E_2 - 3 S] ~.
\]
\end{lem}
\begin{proof}
Note that by the Brunn--Minkowski inequality (e.g. \cite{GardnerSurveyInBAMS}), the function $r \mapsto \log \mu_K\set{|x| \leq r}$ is concave. Indeed, denoting $K_\infty := \Real \times K$ and $L_\infty = \set{(r,x) \in \Real \times \Real^n \;  ; \; r\geq 0 \; , \; |x| \leq r}$, then $\mu_K\set{|x| \leq r} = \vol(K_+ \cap H_r) / \vol(K) $, where $K_+$ is the convex set $K_\infty \cap L_\infty$ and $H_r$ is the hyperplane $\set{r} \times \Real^n$ in $\Real^{n+1}$. Consequently, the function $r \mapsto \brac{\mu_K\set{|x| \leq r}}^{1/n}$ is concave, and in particular, the asserted log-concavity of  $r \mapsto \mu_K\set{|x| \leq r}$ follows. 

Next, note that by Chebyshev's inequality:
\[
\mu_K\set{|x| \leq E - 2S} \leq \frac{1}{4} \;\; ,\;\; \mu_K\set{|x| \leq E + 2S} \geq \frac{3}{4} ~.
\]
The log-concavity of $r \mapsto \mu_K\set{|x| \leq r}$ consequently implies:
\[
\mu_K\set{|x| \leq E - 2S - t} \leq \frac{1}{4} \exp\brac{- \frac{\ln(3)}{4} \frac{t}{S}} \;\;\; \forall t  \geq 0 ~,
\]
or equivalently:
\[
\mu_K\set{|x| \leq r} \leq \frac{1}{4} \exp \brac{- \frac{\ln(3)}{4} \brac{\frac{E-r}{S} - 2} } \;\;\; \forall r \leq E - 2 S ~.
\]
Observing that $E_2^2 = E^2 + S^2$ and hence $E_2 \leq E + S$, it follows that:
\[
\mu_K\set{|x| \leq r} \leq \frac{1}{4} \exp\brac{ 3\frac{\ln(3)}{4} } \exp \brac{- \frac{\ln(3)}{4} \frac{E_2-r}{S} } \;\;\; \forall r \leq E_2 - 3 S ~,
\]
and the assertion of the lemma follows by direct numerical inspection.
\end{proof}

\begin{proof}[Proof of Theorem \ref{thm:main-calc}]
First, we may assume that $E_2 \geq 3 S$, since otherwise the assertion follows trivially with $C = 3/2$. 
Next, since $Z_K(w) / w$ is decreasing, it is enough to prove the assertion for $w_0 = c_1 / (\alpha E_2 S)$, where $c_1$ is the constant from the above Lemma and $\alpha := 2$. 
Integrating by parts, evaluating separately the integral over the ranges $[0,E_2 - 3 S]$ and $[E_2 - 3 S,\infty)$, and using the Lemma, we have:
\begin{eqnarray*}
& & \int \exp(-\frac{w_0}{2} |x|^2) d\mu_K(x)  = \int_0^\infty w_0 r \exp(-\frac{w_0}{2} r^2) \mu_K\set{|x| \leq r}  dr \\
& \leq & \exp\brac{-\frac{w_0}{2} (E_2 - 3 S)^2}  + \int_0^{E_2 - 3 S} w_0 r \exp(-\frac{w_0}{2} r^2) \exp\brac{-c_1 \frac{E_2-r}{S}} dr ~.
\end{eqnarray*}
To evaluate the second term above, denoted $T_2$, we use the change of variables $r = \frac{c_1}{w_0 S} - t = \alpha E_2 - t$:
\begin{eqnarray*}
T_2 & = & \exp\brac{-c_1 \frac{E_2}{S} + \frac{c_1^2}{2 w_0 S^2}} \int_{0}^{E_2 - 3S} w_0 r \exp\brac{- \frac{1}{2} \brac{\sqrt{w_0} r - \frac{c_1}{\sqrt{w_0} S}}^2} dr    \\
&\leq & \exp\brac{-c_1 \frac{E_2}{S} + \frac{c_1^2}{2 w_0 S^2}}  \int_{E_2(\alpha - 1) + 3 S}^\infty \brac{\alpha w_0 E_2 - w_0 t} \exp(-\frac{w_0}{2} t^2) dt ~.
\end{eqnarray*}
Since $\alpha \geq 1$, we estimate the latter expression by:
\begin{eqnarray*}
& \leq  & \exp\brac{-c_1 \frac{E_2}{S} + \frac{c_1^2}{2 w_0 S^2}}  \brac{\frac{\alpha E_2}{E_2 (\alpha-1) + 3S}  - 1} \int_{E_2(\alpha - 1) + 3 S}^\infty w_0 t \exp(-\frac{w_0}{2} t^2) dt \\
& \leq & \exp\brac{-c_1 \frac{E_2}{S} + \frac{c_1^2}{2 w_0 S^2}} \frac{1}{\alpha-1} \exp\brac{ - \frac{w_0}{2} (E_2(\alpha - 1) + 3 S)^2 } ~.
\end{eqnarray*}
Recalling that $w_0 = c_1 / (\alpha E_2 S)$, this is the same as:
\[
= \exp\brac{c_1 \frac{E_2}{S} \brac{\frac{\alpha}{2} - 1}} \frac{1}{\alpha-1} \exp\brac{ - \frac{w_0}{2} (E_2(\alpha - 1) + 3 S)^2 } ~.
\]
It is easy to check that, in fact, for any $\alpha > 0$:
\[
 \exp\brac{c_1 \frac{E_2}{S} \brac{\frac{\alpha}{2} - 1}} \exp\brac{ - \frac{w_0}{2} (E_2(\alpha - 1))^2 } = \exp(-\frac{w_0}{2} E_2^2) ~,
 \]
 but this is most apparent for $\alpha = 2$ (which satisfies the requirement $\alpha \geq 1$ used above). Consequently:
 \[
 T_2 \leq  \frac{1}{\alpha - 1} \exp\brac{-\frac{w_0}{2} E_2 ^2} \leq \frac{1}{\alpha - 1} \exp\brac{-\frac{w_0}{2} (E_2 - 3 S)^2} ~,
 \] 
 and we conclude that for $w_0 =  c_1 / (2 E_2 S)$:
 \[
 \int \exp(-\frac{w_0}{2} |x|^2) d\mu_K(x) \leq 2 \exp\brac{-\frac{w_0}{2} (E_2 - 3 S)^2}  ~.
 \]
 It follows that:
 \begin{eqnarray*}
 \frac{Z_K(w_0)}{w_0} &=& -\frac{1}{w_0} \log\brac{ \int \exp(-\frac{w_0}{2} |x|^2) d\mu_K(x)} \\
 &\geq & \frac{1}{2} (E_2 - 3 S)^2 - \frac{2 \log(2)}{c_1} E_2 S \geq \frac{1}{2} E_2^2 - \brac{3 + \frac{\log(4)}{c_1}} E_2 S ~.
 \end{eqnarray*}
The assertion now follows from the monotonicity of $w \mapsto Z_K(w)/w$ with $c = c_1/2$ and $C = 3 + \frac{\log(4)}{c_1}$.
\end{proof}

\begin{rem}
To appreciate the delicate nature of the bound (\ref{eq:main-calc}) obtained in Theorem \ref{thm:main-calc}, note that the requirement that $\alpha > 1$ was crucially used in the proof. 
In particular, we do \emph{not} know how to extend (\ref{eq:main-calc}) to the following statement, which is perhaps suggested by the heuristic argument outlined in the Introduction:
\[
\forall c > 0 \;\;\; \exists C > 0 \;\;\; \text{ such that } \;\;\;
w \leq \frac{c}{E_2 S} \;\;\; \Rightarrow \;\;\; \frac{Z_K(w)}{w} \geq \frac{1}{2} E_2^2 - C E_2 S ~.
\]
Furthermore, note that our proof above crucially relied on the estimate $\int_a^\infty \exp(-\frac{t^2}{2}) dt \leq \frac{1}{a} \exp(-\frac{a^2}{2})$ valid for positive $a > 0$, and that the slightly rougher bound $\int_a^\infty \exp(-\frac{t^2}{2}) dt \leq C \exp(-\frac{a^2}{2})$ would have incurred an extraneous logarithmic term in the final result: $\frac{Z_K(w_0)}{w_0} \geq \frac{1}{2} E_2^2 - C E_2 S \log(E_2 / S)$. 
\end{rem}

\begin{rem}
A more delicate analysis of the proof above (e.g. using $\alpha = 5$) reveals that we may in fact assert the existence of a constant $c > 0$ so that:
\[
w \leq \frac{c}{E_2 S} \;\;\; \Rightarrow \;\;\; \frac{Z_K(w)}{w} \geq \frac{1}{2} (E_2 - 3 S)^2 ~.
\]
\end{rem}

\section{Sharpness of Free Energy Estimate} \label{sec:4}

Before concluding, we observe that the lower bound on the free energy we obtained in Theorem \ref{thm:main-calc} is in fact sharp in the following sense:

\begin{prop}[Free Energy Upper Bound] \label{prop:sharp}
There exist two universal constants $c , C > 0$ so that for any convex bounded domain $K$ in $\Real^n$:
\[
w \geq \frac{C}{E_2 S} \;\;\; \Rightarrow \;\;\; \frac{Z_K(w)}{w} \leq \frac{1}{2} E_2^2 - c E_2 S ~.
\]
\end{prop}

The proof is immediate given the following lemma, asserting a reverse Chebyshev-type inequality for log-concave measures, which may be of independent (technical) interest:
\begin{lem} \label{lem:reverse-Cheb}
Given a log-concave probability measure $\mu$ on Euclidean space $(\Real^n,|\cdot|)$, let $X$ be a random vector distributed according to $\mu$ and set $E = \E(|X|)$ and $S = \S(|X|)$. There exists a universal constant $c_0  > 0$ so that: 
\begin{equation} \label{eq:reverse-Cheb}
\P(|X| \leq E - c _0 S) \geq c_0 ~.
\end{equation}
\end{lem}

\begin{proof}[Proof of Proposition \ref{prop:sharp}]
Since $\mu_K$ is log-concave and $E \leq E_2$, we know by the previous lemma that:
\[
\mu_K \set{|x| \leq E_2 - c_0 S} \geq c_0 ~.
\]
Consequently, we may estimate:
\[
\int \exp(-\frac{w}{2} |x|^2) d\mu_K(x) \geq c_0 \exp(-\frac{w}{2} (E_2 - c_0 S)^2) ~, 
\]
and hence:
\[
\frac{Z_K(w)}{w} \leq \frac{1}{2} (E_2 - c_0 S)^2 + \frac{\log(1/c_0)}{w} \;\;\; \forall w > 0 ~.
\]
It follows that if $C > 0$ is chosen large enough, then:
\[
w \geq \frac{C}{E_2 S} \;\;\; \Rightarrow \;\;\; \frac{Z_K(w)}{w} \leq \frac{1}{2} E_2^2 - c E_2 S ~,
\]
where:
\[
c := 2 c_0 - c_0^2 - \frac{\log(1/c_0)}{C} > 0 ~.
\]
\end{proof}

It remains to establish Lemma \ref{lem:reverse-Cheb}. Surprisingly, we have not found a reference for it in the literature, perhaps due to the fact that the inequality (\ref{eq:reverse-Cheb}) goes in the opposite direction to the standard large-deviation or small-ball estimates for $|X|$, and so for completeness we provide a proof. It is clear for experts that the statement of the lemma basically reduces to the one-dimensional case using the localization technique, and so our task is ultimately one-dimensional. However, to make the argument as short as possible, we prefer to rely on the available tools in the literature, and do not make any attempts to obtain good numerical constants. We therefore turn to the known reverse-H\"{o}lder  Khinchine-type inequalities between positive and negative moments of $f(X) = |X|^2 - E^2$. Such inequalities between positive moments of \emph{semi-norms} are classical and readily follow by Borell's Lemma \cite{Borell-logconcave} (see e.g. \cite{Milman-Schechtman-Book}); these have been extended to $p$-th order moments for $p=0$ and $p \in (-1,0)$ in \cite{LatalaZeroMomentKhinchine,Guedon-extension-to-negative-p}. However, we require such moment inequalities for the more general function $f$, and thus turn to the reverse-H\"{o}lder inequalities initiated by Bourgain in \cite{Bourgain-LK} for polynomials, and extended using the localization technique by Bobkov\cite{BobkovPolynomials}, Carbery and Wright \cite{CarberyWrightPolynomials} and Nazarov, Sodin and Volberg \cite{NazarovSodinVolbergPolynomials}. For more on this topic, we refer to the excellent survey paper of Fradelizi \cite{FradeliziUltimateKhinchine}, which also extends these results even further.

\begin{thm}[\cite{NazarovSodinVolbergPolynomials,FradeliziUltimateKhinchine,BobkovNazarovUltimateKhinchine}] \label{thm:reverse-Holder}
Let $P$ denote a degree $d$ polynomial on $\Real^n$. Then for any $-1/d < q \leq  p < \infty$ and random-vector $X$ distributed according to a log-concave measure on $\Real^n$:
\[
\norm{P(X)}_{p} \leq C(p,q,d) \norm{P(X)}_{q} ~,
\]
for some (explicit) constant $C(p,q,d)>0$ depending solely on its arguments.
\end{thm} 

\noindent
Here and below we use $\norm{Y}_{p}$ to denote $\E(|Y|^p)^{\frac{1}{p}}$ (with the usual interpretation when $p=0$, which in any case we will not require). 

\begin{proof}[Proof of Lemma \ref{lem:reverse-Cheb}]
By the union bound and Markov-Chebyshev inequality:
\begin{eqnarray*}
\P( \abs{ \abs{X} - E} \leq c_1 S ) &\leq & \P(|X| \geq 8 E ) + \P(\abs{\abs{X}- E} \leq c_1 S  \wedge |X| \leq 8 E ) \\
& \leq  & 1/8 + \P(\abs{|X|^2 - E^2} \leq 9 c_1 S E ) ~.
\end{eqnarray*}
The function $f(x) = |x|^2 - E^2$ is a degree two polynomial, and so applying Theorem \ref{thm:reverse-Holder} with $q = -1/4$ and $p=2$, preceded by Chebyshev's inequality, we have:
\[
\P(\abs{|X|^2 - E^2} \leq 9 c_1 S E ) \leq (9 c_1 S E)^{1/4} \norm{f(X)}_{-1/4}^{-1/4} \leq C' (9 c_1 S E)^{1/4} \norm{f(X)}_{2}^{-1/4} ~. 
\]
Since:
\[
\norm{f(X)}_{2} \geq E \norm{|X| - E}_{2} = E S ~,
\]
we see that choosing the constant $c_1 > 0$ small enough, we can ensure that:
\[
\P( \abs{ \abs{X} - E} \leq c_1 S) \leq 1/4 ~.
\]

Now set $\eps := \P(\abs{X} \leq E - c_1 S)$, and observe that:
\begin{eqnarray} 
\nonumber 0 =\E(|X| - E) &\geq& \brac{\P(\abs{X} \geq E + c_1 S) - P(\abs{X} \leq E + c_1 S)} c_1 S - \int_{c_1 S}^{E} \P(\abs{X} \leq E - r) dr \\
\label{eq:almost} & \geq & S \brac{(1/2 - 2 \eps) c_1 - \int_0^\infty \P(\abs{X} \leq E - (c_1 + t) S) dt } ~.  
\end{eqnarray}
Since $\P(\abs{X} \leq E + 2 S) \geq 3/4$ by Chebyshev's inequality, we may argue exactly as in the proof of Lemma \ref{lem:lc} that the log-concavity of the function $r \mapsto \P(\abs{X} \leq r)$ ensures that:
\[
\P(\abs{X} \leq E - (c_1+t) S) \leq \eps \exp\brac{- \frac{\log 3/4 - \log \eps}{2+c_1} t} \;\;\; \forall t \geq 0 ~.
\]
Plugging this into (\ref{eq:almost}), we obtain:
\[
0 \geq (1/2 - 2 \eps) c_1 - \eps \frac{2 + c_1}{\log 3/4 - \log \eps} ~.
\]
However, this is clearly impossible if $\eps \geq 0$ is too small, and so $\eps$ must be bounded from below by some universal constant $c_2 > 0$. Setting $c_0 = \min(c_1,c_2)$, the asserted claim follows. 
\end{proof}

\setlinespacing{0.82}
\setlength{\bibspacing}{2pt}

\bibliographystyle{plain}
\bibliography{../../ConvexBib}

\end{document}